\documentclass[12pt,reqno]{amsart}
\usepackage{amsmath,amsfonts,euscript,amscd,amsthm,amssymb,upref,graphics,color,verbatim}
\usepackage[normalem]{ulem}
\usepackage{enumitem}
\usepackage{graphicx}
\usepackage{epsfig}
\usepackage[all]{xy}

\theoremstyle{plain}
\newtheorem{theorem}{Theorem}
\newtheorem{proposition}[theorem]{Proposition}
\newtheorem{lemma}[theorem]{Lemma}
\newtheorem{corollary}[theorem]{Corollary}

\theoremstyle{definition}
\newtheorem{definition}[subsubsection]{Definition}
\newtheorem{example}[subsubsection]{Example}
\newtheorem{remark}[subsubsection]{Remark}

\def\acknowledgment{\par\addvspace{17pt}\small\rmfamily
\trivlist\if!\ackname!\item[]\else
\item[\hskip\labelsep
{\bfseries\ackname}]\fi}

\def\C{\mathbb{C}}
\def\R{\mathbb{R}}

\def\N{\mathbb{N}}
\def\D{\mathbb{D}}

\def\P{\mathbb{P}}
\def\cS{\widehat{\Omega}}

\def\cJ{\mathcal{J}_R}

\newcommand{\hC}{\widehat{\mathbb{C}}}

\title{Reduced dynamical systems}
\author{Luka Boc Thaler and Uro\v s  Kuzman}

\begin{document}

\address{L. Boc Thaler: Faculty of Education, University of Ljubljana, SI--1000 Ljubljana, Slovenia. Institute of Mathematics, Physics and Mechanics, Jadranska 19, 1000 Ljubljana, Slovenia.} \email{luka.boc@pef.uni-lj.si}
\address{ U. Kuzman: Faculty of Mathematics and Physics\\
University of Ljubljana\\Slovenia -and-   Institute of Mathematics, Physics and Mechanics\\ Ljubljana
Slovenia} \email{uros.kuzman@fmf.uni-lj.si}

\begin{abstract}
We consider the dynamics of complex rational maps on $\hC$. We prove that, after reducing their orbits to a fixed number of positive values representing the Fubini-Study distances between finitely many initial elements of the orbit and the origin, ergodic properties of the rational map are preserved.
\end{abstract}

\maketitle
\tableofcontents

\section{Introduction}

In this paper we investigate dynamical properties of complex rational maps that are preserved after reducing their orbits to a finite number of real values.  Our work is  motivated by the paper of Forn\ae ss and Peters \cite{FP}, in which they prove that, in the case of a non-exceptional polynomial, one can recover its topological and measure theoretical entropy from the real parts of finitely many elements in every orbit. This result was generalized further to all polynomials by the first named author \cite{Bo}. In the present paper we deal with complex rational maps defined on the Riemann sphere $\hC$. Since there is no natural value that could be assigned to the real part of $\infty$, we instead use the Fubini-Study distance between the origin and a given element of the orbit. Our goal is to determine for which complex rational maps the two above mentioned entropies are preserved after such reduction.

Let us recall the definition of the Fubini-Study distance on the Riemann sphere.
For $z,w \in \C\P^1$ the normalized form of the distance is given by
$$
d_{FS}(z,w)=\frac{2}{\pi}\arccos \frac{|<z,w>|}{||z||\cdot||w||}.
$$
In this set up $0\in\hC$ corresponds to the point $[0:1]$ and $\infty\in\hC$ corresponds to $[1:0]$. Hence, for $z\in\hC$ we can write the following expressions
$$
d_{FS}(0,z)=\frac{2}{\pi}\arccos\sqrt{\frac{1}{1+|z|^2}}
$$
and $d_{FS}(0,\infty)=1$. Note that $d_{FS}(0,z)=d_{FS}(0,w)$ if and only if $|z|=|w|$. Thus, given $\rho\in[0,1)$ the level sets $\Omega_r=\{z\in  \hC\mid d_{FS}(0,z)=\rho\}$ agree with circles $\partial\D_r\subset\C$ of radius $r=\tan{\frac{\pi \rho}{2}}$ and centered at the origin. We call such circles the \emph{prime circles}. Moreover, note that $\Omega_0=\{0\}$ and $\Omega_1=\{\infty\}$ for $r=0$ and $r=1$ respectively. 

Let $R$ be a complex rational map of degree $d\geq 2$ and let $R^n$ denote its $n^{th}$ iterate. Given $z_0\in\hC$ we define $z_n=R^n(z_0)$ and we call $(z_n)$ the orbit of $z_0$. Further, we consider the sequence of Fubini-Study distances $(d_{FS}(0,z_n))_{n\geq 0}$. We prove that such a sequence of real values is completely determined by its first $N=N(R)\geq 0$ elements. 

\begin{lemma}\label{reduction2} Let $R$ be a rational map of degree $d\geq 2$. There exists $N\in\mathbb{N}_0$ such that if   
$d_{FS}(0,z_n)=d_{FS}(0,w_n)$ for all $n\leq N$, then $d_{FS}(0,z_n)=d_{FS}(0,w_n)$ for every $n\in\mathbb{N}_0$. 
\end{lemma}
\noindent Let $N$ be as in lemma above. We define $\Phi:\hC\rightarrow [0,1]^{N+1}$ to be the map given by
$$\Phi(z_0):=(d_{FS}(0,z_0),\ldots,d_{FS}(0,z_N)).$$ 
Then the action of $R$ can be pushed down to a compact subset $\cS:=\Phi(\hC)\subset [0,1]^{N+1}$, i.e. there exists a map $Q:\cS\rightarrow \cS$ such that the following diagram commutes. 
\begin{equation*}
  \xymatrix@R+2em@C+2em{
  \hC \ar[r]^-{R} \ar[d]_-{\Phi} & \hC \ar[d]^-{\Phi} \\
  \cS \ar[r]_-Q & \cS
  }
 \end{equation*}
In particular, $Q$ is given by 
$$
Q(d_{FS}(0,z_0),\ldots,d_{FS}(0,z_N)):=(d_{FS}(0,z_1),\ldots,d_{FS}(0,z_{N+1})).
$$ 
Even though the map $\Phi$ is never an embedding, many properties of the dynamical system $(\hC, R)$ can be observed by analyzing the reduced system $(\cS, Q)$. In this paper we focus only on the ergodic properties of the dynamical systems.

Let $\cJ$ denote the {\it Julia set} of $R$. By classical results of Lyubich and Ma\~{n}\'{e} \cite{Lyu, Mane} every rational map $R$ of degree $d\geq 2$ admits a unique invariant, ergodic probability measure $\mu_R$ on $\hC$ of maximal entropy $\log d$. Moreover, the equidistribution property for repelling periodic points of $R$ implies that $\text{supp}\mu_R=\cJ$.   

Let $\nu_R:=\Phi_*(\mu_R)$ be the corresponding measure on the set $\cS$ and let $\D_r$ denote the disk of radius $r$ centered at the origin. The following is our main result. 

\begin{theorem}\label{main1} Let $R$ be a rational map of degree $d\geq 2$. If $\cJ$ is not contained in a prime circle, then $\nu_R$ is the unique invariant, ergodic measure of maximal entropy $\log d$ on $\cS$.
\end{theorem}
\noindent We also prove that if $\cJ$ is contained in some prime circle, then both, the topological entropy of $Q$ and the measure theoretical entropy of $\nu_R$, are equal to zero (see Proposition \ref{zero}). Therefore, we call such maps $R$ strongly exceptional.

The paper is organized as follows. In \S 2 we recall some basic facts concerning ergodic theory and prove Lemma \ref{reduction2}. In \S 3 we introduce the notion of exceptional maps and study the semi-analytic set of mirrored points $M$, i.e. the set of points $z\in M$ for which there exists $w\neq z$ such that $\Phi(z)=\Phi(w)$. We prove that either $M$ contains a dense open subset of $\hC$ or else $\dim_{\R}M=1$ (see Theorem \ref{mirrored}). In \S 4 we give the proof of Theorem \ref{main1}. Moreover, we prove Lemma \ref{circle} which may be of independent interest. It states that whenever $\mu_R$ puts a mass on a one dimensional semi-analytic set, the Julia set $\cJ$ is contained an invariant circle of $\hC$. 

Finally note that the techniques used in this paper are quite different from those in \cite{FP}. In particular, we deal with proper real analytic maps and sub-analytic sets instead of real algebraic maps and semi-algebraic sets. 

\medskip
{\it Acknowledgements}: The research was initiated during the stay of both authors at the University of Oslo, Spring 2017. They want to thank prof. Erlend F. Wold for his hospitality. The authors would also like to thank Han Peters for helpful discussions on this topic.

\section{Preliminary results}

\subsection{Entropy} Let $(X,\rho)$ be a compact metric space and $F\colon X\to X$ a continuous map. For $\varepsilon>0$ we define a $(n,\varepsilon)$-ball centered in $x\in X$ as
$$B(x,\varepsilon,n)=\{y\in X\mid \rho(F^k(x),F^k(y))<\varepsilon,\; 0\leq k<n\}.$$
Let $N(n,\varepsilon)$ denote the maximal number of pairwise disjoint $(n,\varepsilon)$-balls in $X$ and define
$$
H_{\varepsilon}=\limsup_{n\in\mathbb{N}} \frac{1}{n}\log{N(n,\varepsilon)}.
$$
The \emph{topological entropy} of $F$ is defined as
$$
h_{top}(F)=\lim_{\varepsilon\rightarrow 0}H_{\varepsilon}.
$$
Further, let $\lambda$ be a probability measure on $X$ and define
$$
h_{\lambda}(F,x,\varepsilon)=\liminf_{n\in\mathbb{N}} -\frac{1}{n}\log \lambda(B(x,\varepsilon,n))
$$
and
$$
h_{\lambda}(F,x)=\sup_{\varepsilon>0}h_{\lambda}(F,x,\varepsilon).
$$
If $\lambda$ is an invariant measure then $h_{\lambda}(F,x)\geq h_{\lambda}(F,F(x))$ and if $\lambda$ is also ergodic this function is constant $\lambda$-almost everywhere.  The \emph{measure theoretic entropy} $h_{\lambda}(F)$ is then defined to be this constant. Note that it is independent of the metric. In fact it is a
topological invariant.

Let $X$ and $Y$ be compact metric spaces and let $R\colon X\rightarrow X$ and $Q\colon Y\rightarrow Y$ be continuous maps which are semi-conjugated. That is, there exist a continuous, surjective map $\Phi\colon X\rightarrow Y$ such that the following diagram commutes. 
 \begin{equation*}
  \xymatrix@R+2em@C+2em{
  X \ar[r]^-{R} \ar[d]_-{\Phi} & X \ar[d]^-{\Phi} \\
  Y \ar[r]_-Q & Y
  }
 \end{equation*}
Moreover, we assume that the fibers of $\Phi$ are finite and that their cardinality is uniformly bounded from above, i.e. there exist $C_R\in\mathbb{N}$ such that:
\begin{equation}\label{card}\# \Phi^{-1}(y)\leq C_R, \;\; \forall y\in Y.\end{equation}
In such a set-up we have the following proposition.   
\begin{proposition}\label{proposition1}
Let $R,$ $Q$ and $\Phi$ be defined as above and let $\mu$ be an invariant, ergodic measure on $X$. If the condition (\ref{card}) is satisfied then the following statements hold:
\begin{itemize}
\item[(a)] The topological entropies of maps $Q$ and $R$ agree, t.i., $h_{top}(R)=h_{top}(Q).$
\item[(b)] The measure $\nu=\Phi_*(\mu)$ is invariant and ergodic on $Y$. 
\item[(c)] Suppose $h_{\mu}(R)=h_{\nu}(Q)$. If $\mu$ is the unique measure of maximal entropy on $X$ then $\nu$ is also the unique measure of maximal entropy on $Y$.
\end{itemize}
\end{proposition}
\noindent 
\begin{proof} This proposition is a summary of known results which were also proven in \cite{FP}  in a slightly less general form. Therefore we will only sketch the proof.  

 (a) In general, the semi-conjugacy of $R$ and $Q$ implies only $h_{top}(R)\geq h_{top}(Q)$. However, in our case the two topological entropies agree due to condition (\ref{card}), e.g. see \cite[Theorem 4.1.15]{D}.  
 
 (b) This statement is proved in  \cite[Lemma 4.2.]{FP}. Note that it remains valid even in the case of unbounded fibers. That is, we do not need condition (\ref{card}) to prove it.
  
  (c) As before semi-conjugacy of $R$ and $Q$ implies  $h_{top}(R)\geq h_{top}(Q)$. Combining  (b) with the variational principle we obtain 
\begin{equation}\label{entropy}
h_\nu(Q)\leq \sup_\lambda h_{\lambda}(Q)\leq h_{top}(Q)\leq h_{top}(R)=h_\mu(R),
\end{equation}
where the supremum is taken over all invariant ergodic measures on the set $Y$,  hence in every semi-conjugate system  we have $h_{\mu}(R)\geq h_{\nu}(Q)$. Suppose that $h_{\mu}(R)=h_{\nu}(Q)$, then it follows immediately from \eqref{entropy} that $\nu$ is the measure of maximal entropy on $Y$. To prove the uniqueness let  $\sigma\neq \nu$ be any other invariant ergodic probability measure on $Y$. By \cite[Lemma 4.9]{FP} the measure $\sigma$ is the push forward of an invariant ergodic probability measure $\tau$ on $X$. Since $\sigma\neq \nu$ and since $\mu$ is unique measure of maximal entropy we have $\tau \neq\mu$ and $h_\tau(R)<h_{\mu}(R)$. Since $\sigma$ is a push forward of $\tau$ it  follows that 
$h_{\sigma}(Q)\leq h_\tau(R)$, hence $h_{\sigma}(Q)<h_{\nu}(Q)$ which completes the proof of uniqueness.

\end{proof}
\begin{remark}\label{remark1} In what follows, we sometimes consider maps $\Phi$ for which (\ref{card}) is satisfied everywhere but in a single infinite fiber $\Phi^{-1}(y_0)$. However, note that the claim  $(c)$ from Proposition \ref{proposition1} can still be applied. Indeed, one can check that \cite[Lemma 4.9]{FP} remains valid if the fibers $\Phi^{-1}(y)$ satisfy the condition (\ref{card}) for $\nu$-almost every $y\in Y$. These details are left to the reader.
\end{remark}

\subsection{Proof of Lemma \ref{reduction2}}

 The map $z\rightarrow d_{FS}(0,z)$ is real analytic on $\hC$, hence given   $n\in\mathbb{N}_0$ the map $\Phi_n:\hC\times\hC\rightarrow \R$ defined as
$$
\Phi_n(z,w)=d_{FS}(0,R^n(z))-d_{FS}(0,R^n(w)).
$$ 
is real analytic as well. Furthermore, we define 
$$
Z_n=\{(z,w)\in \hC\times\hC\mid \Phi_0=\Phi_1=\ldots=\Phi_n=0 \}.
$$
Note that $Z_{n}\supset Z_{n+1}$ for all $n$. By  \cite[Theorem I.9.]{F} the ring of global real analytic functions on a compact real analytic manifold is noetherian (see also the last paragraph in \cite{R}). Hence, there exists $N\in\mathbb{N}_0$ such that $Z_{n}= Z_{n+1}$ for all $n\geq N$.

\section{Mirrored points}
In this section we investigate the fibers of $\Phi$ and the corresponding data reduction. Let us begin with the following elementary example.  

\begin{example} Let $R(z)=z^{d}$ where $d\geq 2$. Then $|z|=|w|$ implies $|R(z)|=|R(w)|$ therefore $N=0$ and $Q\colon [0,1]\to[0,1]$ can be computed explicitly
$$Q(x)=\frac{2}{\pi}\arctan\left(\tan\left(\frac{\pi x}{2}\right)\right)^{d}.$$ 
Since $Q$ is a homeomorphism of the unit interval we have $h_{top}(Q)=0$. In contrast, we know from \cite{Lyu, Mane} that $h_{top}(R)=\log d$. $\clubsuit$ 
\end{example}

\subsection{Exceptional maps} As seen above, map $\Phi$ compresses the behavior along the prime circles. Hence, in order to preserve its entropy, we have to exclude cases in which the loss of information would be too large. That is, in order to apply Proposition \ref{proposition1}, we have to assure that all fibers $\Phi^{-1}(y)$, $y\in \cS$ (except at most one) are finite and have their cardinality uniformly bounded.

The following lemma shows that this is true for all $R$ that omit the special case presented in the above example. The proof is based on the fact that the only rational maps $R\colon \hC\to\hC$ satisfying $R(\partial\mathbb{D})\subset\partial\mathbb{D}$ are the \emph{finite Blaschke products} \begin{equation}\label{form}
R(z)=e^{i\varphi}z^{\ell}\prod^{n}_{k=1}\frac{z-a_k}{1-\bar{a}_kz},
\end{equation}
where $a_k\in \mathbb{C}^*\backslash \partial\D$, $\ell\in \mathbb{Z}$ and $\varphi\in\mathbb{R}$.

\begin{lemma}\label{R-invariant} Let $R$ be a rational map of degree $d\geq 2$. If $R$ maps two distinct prime circles into prime circles, it is of the form $R(z)=az^{\pm d}$, $a\in\mathbb{C}^*$.  
\end{lemma}
\begin{proof}
Assume that $R$ maps prime circles of radii $s_1>r_1>0$ into prime circles with radii $s_2>0$ and $r_2>0$ respectively. Without loss of generality we can assume that $r_1=r_2=1$. That is, $R$ is of the form 
$$R(z)=e^{i\varphi}z^{\ell}\prod^{n}_{k=1}\frac{z-a_k}{1-\bar{a}_kz},$$
for some $a_k\in \mathbb{C}^*\backslash \partial\D$, $\ell\in \mathbb{Z}$ and $\varphi\in\mathbb{R}$. Furthermore, the same is true for  
$$\frac{1}{s_2}R(s_1\cdot z)=e^{i\theta}z^{\tilde{\ell}}\prod^{m}_{k=1}\frac{z-b_k}{1-\bar{b}_kz}.$$
However, $\ell=\tilde{\ell}$ and $n=m$ since the function $R$ was only rescaled. Thus
$$\frac{1}{s_2}R(s_1\cdot z)=e^{i\theta}z^{^\ell}\prod^{n}_{k=1}\frac{z-b_k}{1-\bar{b}_kz}=\frac{1}{s_2}e^{i\varphi}(s_1z)^{^\ell}\prod^{n}_{k=1}\frac{s_1z-a_k}{1-\bar{a}_ks_1z}.$$
 
We divide this relation with $z^\ell$ and observe its zeroes. It is evident that after reordering the products one can assume that $a_k=b_ks_1$. Furthermore, comparing the poles one sees that $s_1=1$. This yields contradiction unless $n=0$.
\end{proof}

The lemma and the example above indicate that the case $R(z)=az^{\pm d}$ is very special and has to be treated separately. Moreover, in the spirit of the Remark \ref{remark1}, we pay a special attention to the following maps as well.
  
\begin{definition}
A rational map $R$ is \textit{exceptional} if $R(\D_r)\subset \D_r$ for some $r>0$, i.e.  $R(rz)/r$ is a finite Blaschke product.  
\end{definition}
\noindent The following proposition establishes the bound (\ref{card}) from \S 2.1. As pointed out above this bound is violated in precisely one infinite fiber when $R(z)\neq az^{\pm d}$ is exceptional. We introduce some additional terminology. We say that $z\in \hC$ is \emph{mirrored} by $w\in\hC$ if $\Phi(z)=\Phi(w)$ and $w\neq z$, that is, they belong to the same fiber. The set of all such $w\neq z$ is called the \emph{mirrors of} $z$. Finally, $z$ \emph{is mirrored} if it has at least one mirror. 

\begin{proposition}\label{cardinality}
Let $R$ be a rational map of degree $d\geq 2$. The following holds:
\begin{itemize} 
\item[a)] If $R$ is non-exceptional, then $\#\Phi^{-1}(y)\leq 2d^2$ for every $y\in\cS$. 
\item[b)] If $R$ is exceptional and $R(z)\neq az^{\pm d}$, there exists a unique $y_0\in\cS$ such that the fiber $\Phi^{-1}(y_0)$ is infinite and that $\#\Phi^{-1}(y)\leq 2d$ for every $y\in\cS\setminus\left\{y_0\right\}$. 
\end{itemize}
\end{proposition}
\begin{proof} First note that the following statement holds: If $r_1>0$ and $r_2>0$ are distinct and such that $R(\partial\D_{r_1})\not\subset \partial\D_{r_2}$, then the number of points in $\partial\D_{r_1}$ that are mapped into $\partial\D_{r_2}$ is at most $2d$. Indeed, let $\overline{R}$ be the rational map obtained from $R$ by conjugating its coefficients. Since $R(\partial\D_{r_1})\not\subset \partial\D_{r_2}$ the map 
$$z\to R(z)\overline{R}(r_1^2/z)-r_2^2$$
does not vanish identically on $\partial\D_{r_1}.$ Hence, we may extend it as a $2d$-degree rational map that admits at most $2d$ distinct zeroes on $\hat{\mathbb{C}}$. 

Let $R$ be non-exceptional. Note that $z=0$ and $z=\infty$ can not be mirrored. Hence $z\in R^{-1}(\left\{0,\infty\right\})$ admits at most $d-1$ mirrors and $z\in R^{-2}(\left\{0,\infty\right\})$ admits at most $d^2-1$ mirrors. Next, assume that $\left\{z,R(z),R^2(z)\right\}\cap\left\{0,\infty\right\}=\emptyset$ and suppose that $r_0,r_1,r_2>0$ are the radii of prime circles through $z$, $R(z)$ and $R^2(z)$ respectively. If $R(\partial\D_{r_0})\not\subset \partial\D_{r_1}$ then there are at most $2d$ points in $\partial\D_{r_0}$ mapped to $\partial\D_{r_1}$. Thus, $z$ can be mirrored by at most $2d-1$ points. In contrast, if $R(\partial\D_{r_0})\subset \partial\D_{r_1}$ we know that $\partial\D_{r_0}\neq \partial\D_{r_1}$ since there is no $R$-invariant prime circle. Moreover, $R(\partial\D_{r_1})\not\subset \partial\D_{r_2}$ since two prime circles can not be carried into prime circles. Hence at most $2d$ points in $\partial\D_{r_1}$ are mapped into $\partial\D_{r_2}$. Each of them has at most $d$ preimages in $\partial\D_{r_0}$, so there are at most $2d^2-1$ points that mirror $z$.

Let $\partial\D_{r_0}$ be the unique invariant prime circle of an exceptional map $R$. Suppose that $z\in\C^*\backslash \partial\D_{r_0}$ and let $\partial\D_{r_1}$ and $\partial\D_{r_2}$ be the prime circles through $z$ and $R(z)$ respectively. Clearly $\partial\D_{r_0}\neq \partial\D_{r_1}$ and since $R(z)\neq az^{\pm d}$ we also have $R(\partial\D_{r_1}) \not\subset \partial\D_{r_2}$. Hence, there are at most $2d$ points in $\partial\D_{r_1}$ that are mapped to the circle $\partial\D_{r_2}$. 
\end{proof}

In the example above we show that if $R(z)=az^{\pm d}$, one has $h_{top}(Q)=0$. This is due to the fact that the Julia set of such $R$ is contained in a prime circle that is compressed by $\Phi$. Our next aim is to prove that this is true for any map $R$ that satisfies $\cJ\subseteq\partial\mathbb{D}_r$ for some $r>0$. Since $\cJ$ is $R$-invariant and uncountable, it follows by Proposition \ref{cardinality} that in this case we have $R(\partial\D_r)\subseteq \partial\D_r$. Therefore all such maps are exceptional.   

\begin{definition} A rational map $R$ is \textit{strongly exceptional} if $\cJ\subseteq\partial \mathbb{D}_r$ for some $r>0$.
\end{definition}
\noindent Note that such maps were classified in \cite{EvS}. We use this classification in order to prove the following proposition. 

\begin{proposition}\label{zero}If $R$ is a strongly exceptional rational map, we have $h_{top}(Q)=0$.
\end{proposition}
\begin{proof} Our proof is based on the result from \cite{Dow}: Given a continuous endomorphism $F$ defined on a compact metric space $(X,d)$ and satisfying $h_{top}(F)>0$ there exists an uncountable set $E\subset X$ such that for any $x, y\in E$ we have 
$$
\liminf_{n\rightarrow \infty}\frac{1}{n}\sum_{k=1}^nd(F^k(x),F^k(y))=0\quad\text{and}\quad \limsup_{n\rightarrow \infty}\frac{1}{n}\sum_{k=1}^nd(F^k(x),F^k(y))>0.
$$
Let us try to find such set $E$ for our map $F=Q$ and the space $X=\cS$ equipped with the usual Euclidean metric. 

According to \cite[Theorem 2]{EvS} we have to consider two cases. Firstly, a strongly exceptional map $R$ can be represented as a finite Blaschke product with all the zeros belonging to either $\D$ or $\hC\setminus\overline{\D}$. In this case, $\cJ$ is equal to $\partial\D$ or to a Cantor subset of $\partial\D$. However, in both cases for $z\in\hC\setminus \partial \D$ the iterates $R^n(z)$ approach one of at most two points given by Denjoy-Wolff Theorem. 

In contrast, all other strongly exceptional maps $R$ satisfy the following two conditions: $R$ admits a fixed point $w_0\in\partial\D$ whose multiplier belongs to $[-1,1]$; the Julia set $\cJ$ is contained in a closed arc $I\subset \partial\D$ whose interior does not contain $w_0$. However, for all such maps we have $R^n(z)\to w_0$ where $z\in\hC\setminus \partial \D$.

The above consideration implies that the condition
$$\liminf_{n\rightarrow \infty}\frac{1}{n}\sum_{k=1}^nd_{FS}(R^k(z),R^k(w))\neq\limsup_{n\rightarrow \infty}\frac{1}{n}\sum_{k=1}^nd_{FS}(R^k(z),R^k(w))$$
can be fulfilled only if at least one of the points $z$ and $w$ is contained in $\cJ$. However, the map $\Phi$ compresses $\cJ$ into a single point. Hence the condition   
$$\liminf_{n\rightarrow \infty}\frac{1}{n}\sum_{k=1}^n||Q^k(x)-Q^k(y)||\neq\limsup_{n\rightarrow \infty}\frac{1}{n}\sum_{k=1}^n||Q^k(x)-Q^k(y)||$$
is satisfied for at most two points $x,y\in \cS$. That is, $\#E\leq 2$ and $h_{top}(Q)=0$. 
\end{proof}

\subsection{Mirrored set}

As seen above, the mirrored points play important role in the analysis of our reduced dynamical system. Therefore we define the following set 
$$
M:=\{z\in\hC\mid \exists w\neq z: \Phi(w)=\Phi(z) \}.
$$ 
We call it {\it the mirrored set}. Let us prove that it is never empty. 

\begin{lemma}\label{injective} Let $R$ be a rational map of degree $d\geq 2$. Given its critical point $c\in\hC$ there exist distinct points $z,w\in \hC$ arbitrarily close to $c$ and such that $\Phi(z)=\Phi(w)$.
\end{lemma}

\begin{proof}  
It suffices to find $z$ and $w$ near $c$ such that $z\neq w$, $|z|=|w|$ and $R(z)=R(w)$. Suppose $c\notin\left\{0,\infty\right\}$. There exist $k\geq 2$ and a small neighborhood $U_c$ of $c$ such that the map $R|_{U_c}$ acts as a small perturbation of $v+a_k(z-c)^k$, $a_k\neq 0$. Fixing $\epsilon>0$ small enough we may assume that the curve
$$K_{\epsilon}=\left\{R^{-1}(v+\epsilon e^{i\varphi}),\;\; \varphi\in [0,2\pi)\right\}\cap U_c$$
is a small perturbation of a circle centered in $c$ and of radius $\sqrt[k]{\epsilon}$. Hence, given $\theta\in[0,2\pi)$ and $w_{\theta}=v+\epsilon e^{i\theta}$ there are $k$ distinct points $z_1,z_2,\ldots,z_k\in K_\epsilon$ ordered in the counter clockwise direction and such that $R(z_j)=w_{\theta}$. We denote by $J_j$ the open arcs connecting two consecutive points $z_j$ and $z_{j+1}$ or $z_k$ and $z_1$ respectively. Note that the closure of every such arc is mapped into a full circle around $v$. 

For $\epsilon>0$ small enough the set $K_{\epsilon}$ intersects the circle $|z|=c$ in precisely two points, one in the upper-half plane and one in the lower-half plane. Without loss of generality we may assume that these two points are contained in two distinct $J_{j_1}$ and $J_{j_2}$ (if not we perturb $\theta)$. Furthermore, we may assume that $|z_{j_1}|<|z_{j_1+1}|$ and $|z_{j_2}|>|z_{j_2+1}|$. Since each of these sets is mapped into a full circle around $v$, continuity implies that there exist $\zeta_{j_1}\in J_{j_1}$ and $\zeta_{j_2}\in J_{j_2}$ such that $|\zeta_{j_1}|=|\zeta_{j_2}|$ and $R(\zeta_{j_1})=R(\zeta_{j_2})=v+\epsilon e^{i\varphi}$ for some $\varphi\in[0,2\pi)$. This concludes the proof for $c\notin\left\{0,\infty\right\}$.

If $c\in\left\{0,\infty\right\}$, we construct the arcs $J_j$ in an analogous way. Furthermore, we may assume that $|z_{1}|\leq |z_j|$. Hence, by the similar argument as above there exist $\zeta_{1}\in J_{1}$ and $\zeta_{k}\in J_{k}$ such that $|\zeta_{1}|=|\zeta_{k}|$ and $R(\zeta_{j})=R(\zeta_{k})$.
\end{proof}

\begin{remark}\label{curve} Note that the points $\zeta_{j}$ depend continuously on $\epsilon>0$. Hence $M$ always contains a piecewise continuous curve that approaches $c$ when $\epsilon\to 0$. However, the point $c$ needs not to be in $M$.
\end{remark}

We proceed by giving two explicit examples of the set $M$.

\medskip

\begin{example} Let $R(z)=(z-1)^2+1$. Since $R$ has real coefficients, it follows that $|R(z)|=|R(\bar{z})|$ and hence  $\Phi(z)=\Phi(\bar{z})$. Therefore $\C\backslash\R\subseteq M$ (the points $0$ and $\infty$ are never mirrored). Moreover, let $z\in\R$. Given $n\in\N_0$ we can define a polynomial in three real variables $x$, $y$ and $z$:
 $$\phi_n(x,y,z):=|R^n(x+iy)|^2-|R^n(z)|^2.$$
A point $x+iy$ is mirrored by $z$ if and only if $\phi_n(x,y,z)=0$ for all $n\in\N$. In order to find solutions of this system, we can try to determine the basis for the ideal generated by polynomials $\phi_n$. In fact one can compute that the Gr\"{o}bner basis consists of three polynomials
$-y^2z^2(1-z)$, $y^2( 4 + 4 y^2- 8 z +z^2)$ and $3 x +  y^2(3-z^2) - 3 z$. Observe that the intersection of the zero sets of these three polynomials in $\R^3$ is equal to $\{(\zeta,0,\zeta)\mid\zeta\in\R\}\cup\{(\frac{1}{2},-\frac{\sqrt{3}}{2},1),(\frac{1}{2},\frac{\sqrt{3}}{2},1)\} $, which implies that $1$ is the only mirrored point on the real axis. Thus $M=\{1\}\cup \C\backslash\R$. Finally, let us emphasize that in this case $M$ contains a dense open subset of $\hC$, which is also true for every rational map with real coefficients. $\clubsuit$
\end{example} 

\medskip

\begin{example} Let us conjugate $P(z)=iz^2$ by a fractional transform $\varphi(z)=\frac{2(z-1)}{z+2}$: $$R(z)=\varphi^{-1}\circ P\circ\varphi(z).$$ 
Note that $\varphi(\partial\D(-\frac{1}{2},\frac{3}{2}))=i\R$ and that for every $z\in\partial\D(-\frac{1}{2},\frac{3}{2})$ we have $\overline{\phi(z)}=\phi(\overline{z})$. This implies that $\Phi(z)=\Phi(\bar{z})$ for all $z\in\partial\D(-\frac{1}{2},\frac{3}{2})\setminus\left\{-2,1\right\}$. We claim that this is precisely the mirrored set $M$. 

Note that the points from $\D(2,2)$ and $\hC\backslash\overline{\D}(2,2)$ are attracted to fixed points of different modulus, namely $1$ and $-2$. Moreover, a point $z\in\cJ=\partial\D(2,2)$ with $|z|=1$ or $|z|=2$ is not periodic. Therefore, there can be no mirroring between these three sets and we can discuss them separately. Given $z\in\partial\D(2,2)$ the situation is clear. Such a point can only be mirrored by $\bar{z}$. However, this happens only for $\varphi^{-1}(\pm i)\in \D(-\frac{1}{2},\frac{3}{2})$. If $z\in\D(2,2)\setminus\left\{1\right\}$ the iterates $R^n(z)$ move towards $1$. However, close to this point the prime circles are almost identical to parallel lines and $R$ is almost $z\rightarrow \frac{2}{3}i(z-1)^2+1$. 

\begin{figure}[h]
\centering
\includegraphics[width=.25\textwidth]{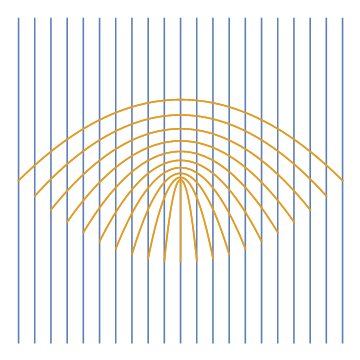}
	\
\caption{The orange curves represent the image of the blue vertical lines under the map $z\rightarrow \frac{2}{3}i(z-1)^2+1$.}
\end{figure}

\noindent In such an approximative setting, a vertical line intersects the image of some other line in a unique point. The only exception is the line $z=1+iy$, $y\in\mathbb{R}$, which is mapped into a ray below the point $1$. Therefore this has to be the line  approximating $\partial\D(-\frac{1}{2},\frac{3}{2})\backslash\left\{1\right\}$ and there can be no other mirrored points near $1$. Next, observe that if $z,w\in \D(2,2)$, we have $R(z)=R(w)$ only when $w=z$ or $w=\overline{z}$. Moreover, $R(z)$ and $R(\overline{z})$ either belong to different components of $\D(2,2)\backslash \partial\D(-\frac{1}{2},\frac{3}{2})$ or else they are both contained in $\partial\D(-\frac{1}{2},\frac{3}{2})$. This implies that given $z\in M\cap\D(2,2)$ there is $n\in\N$ such that $R^n(z)\in \partial\D(-\frac{1}{2},\frac{3}{2})$. Thus $R^{n-1}(z)$ is contained in $\partial\D(-\frac{1}{2},\frac{3}{2})$ or in $[0,4]$. However, the interval $[0,4]$ contains no mirrored points. Therefore, $\partial\D(-\frac{1}{2},\frac{3}{2})\cap (\D(2,2)\setminus\left\{1\right\})$ equals to $M\cap\D(2,2)$. We treat the set $\hC\backslash\overline{\D}(2,2)$ in an analogous way. $\clubsuit$ 
\end{example}

\medskip

According to the examples above, the set $M$ can be of two types: either it contains a dense and open subset of $\hC$ or its dimension is equal to $\dim_\mathbb{R}M= 1$. We prove in the sequel that these are the only two cases that can occur. In order to do this, we present $M$ in the spirit of sub-analytic sets.

The class of sub-analytic sets is generated by images of proper real analytic maps into real analytic manifolds, with respect to operations like finite union, finite intersection and difference. Hironaka \cite{H} proved that every such set $X$ admits a locally finite stratification in which the strata are locally closed, connected real-analytic sub-manifolds in the ambient manifold. This enables us to define the topological dimension $\mathrm{dim}_{\R}X$ which does not depend on the choice of stratification. Furthermore, the points in $X$ are of two types. Regular, if in their vicinity $X$ is an analytic sub-manifold of dimension $\mathrm{dim}_{\R} X$, and singular otherwise. Moreover, it is known that the set of singular points is sub-analytic and that $\mathrm{dim}_{\R}\mathrm{Sing}(X)\leq\mathrm{dim}_{\R}X-1$. However, if $\dim_{\R}\mathrm{Sing}(X)\leq 1$, the set of singular points is semi-analytic as well. The reader is referred to \cite{T, BM, L} for more details.

In our case $X=\cS$ is the image of the proper real analytic map $\Phi:\hC\rightarrow [0,1]^{N+1}$. Hence $\cS$ is a compact sub-analytic and $\dim_{\R}\mathrm{Sing}(\cS)\leq 1$. Therefore, the set $\mathrm{Sing}(\cS)$ is compact and semi-analytic. Let us define the set 
$$
S:=\Phi^{-1}(\mathrm{Sing}(\cS)).
$$
Recall that a map $R(z)\neq az^{\pm d}$ admits at most one infinite $\Phi$-fiber, that is, an invariant prime circle. Therefore, since $\Phi$ is proper and real analytic, the set $S$ is compact, semi-analytic and of real dimension at most one. Moreover, if $R(z)=az^{\pm d}$, we have $S=\{0,\infty\}$ and $M=\hC \backslash S$. Finally, note that $M\subseteq S$ when a generic point is not mirrored. In contrast, $\hC\setminus S\subset M$ if the map $\Phi$ is a multiple cover over the set $\mathrm{Reg}(\cS)$. This gives the following theorem.

\begin{theorem}\label{mirrored} Let $R$ be a rational map of degree $d\geq 2$. Then its mirrored set $M$ is non-empty and semi-analytic. Moreover, it either contains a dense and open subset of the Riemann sphere $\hC$ or else $\dim_{\R}M=1$. 
\end{theorem}
\begin{proof}
Let us denote by $\Gamma$ the diagonal in the set $\hC\times\hC$ and by $\Lambda$ the set  
$$
\Lambda=\{(z,w)\in\hC\times\hC\mid \Phi(w)=\Phi(z)\}.
$$ 
Then $M=\pi_1(\Lambda\backslash \Gamma)$, where  $\pi_1$ denotes the projection to the first component. Since the sets $\Lambda$ and $\Gamma$ are real-analytic, the set $\Lambda\backslash \Gamma$ is semi-analytic. Moreover since $\pi_1$ is proper, it follows that $\pi_1(\Lambda\backslash \Gamma)$ is sub-analytic. By Remark \ref{curve} we have $\dim_{\R} M\geq 1$. 

Since $\Phi$ is a covering map over the set of regular points $\mathrm{Reg}(\cS)$, one of the following two cases has to occur:
\begin{enumerate}
\item[i)] $\Phi^{-1}(\mathrm{Reg}(\cS))\subseteq M$. 
\item[ii)] $M\subseteq \pi_1\left(\overline{\Lambda\backslash \Gamma}\right)=S$.  
\end{enumerate}
In particular, case $i)$ corresponds to $\deg\Phi\geq 2$ and case $ii)$ corresponds to $\deg\Phi=1$. 
Since $\Phi^{-1}(\mathrm{Reg}(\cS))\subseteq\hC$ is open and dense, and since $M$ is sub-analytic it follows that $\hC\backslash M$ is semi-analytic of dimension at most 1. But this implies that $M$ is semi-analytic as well. We have seen above that $S$ is a compact semi-analytic set of dimension at most one, therefore, in the case $ii)$, we have $\dim_\R M=1$ and hence $M$ is semi-analytic.
\end{proof}

\noindent Note that the single infinite fiber of an exceptional map will always be in $S$. Moreover, let us denote by $\mathcal{X}$ the set of points $c\in\hC$ for which there exists a sequence of mirrored pairs $(z_n,w_n)$, i.e. $z_n\neq w_n$ and $\Phi(z_n)=\Phi(w_n)$, that converges to $(c,c)$. Clearly the map $\Phi$ can not be injective in any neighborhood of such a point, hence $\mathcal{X}\subset S$. Furthermore, by Remark \ref{curve} the critical points of $R$ are contained in $\mathcal{X}$ and therefore in $S$. However, as pointed out, a point from $\mathcal{X}$ needs not to be mirrored. Hence, in general, we have $M\neq S$ in the case $i$) and $M\neq\hC\backslash S$ in the case $ii$). We end this discussion with the following corollary that will be used in the proof of Theorem \ref{main1}. It follows directly from the properties of $S$ and the fact that $\Phi$ is a covering map over the set of regular points.

\begin{corollary}\label{cover} Let $U\subset\hC$ be an open neighborhood of $S$. There exists $\epsilon>0$ such that for every $z\in \hC\backslash U$ the set $$B'(z,\epsilon):=\{w\in\hC\mid ||\Phi(w)-\Phi(z)||<\epsilon\}, $$
consists of connected components, where each component is contained in small disk centered
at one of the mirror points of $z$. Moreover these disks are pairwise disjoint and $R$ is one-to-
one on each of them.  
\end{corollary}

Finally, let us prove that, unless $\cJ=\hC$, the mirrored set $M$ of a generic rational map $R$ satisfies the property $\dim_{\R}M=1$.  

\begin{proposition}
Let $R$ be a non-exceptional rational map of degree $d\geq 2$ admitting an attracting periodic point $z_0\in \C^*$ with a non-real multiplier. Assume that $z_0$ is not mirrored by any other non-repelling periodic point or a point belonging to the set $R^{-1}(z_0)$. Then $\dim_\mathbb{R}M= 1$.
\end{proposition}
   
\begin{proof}
By Theorem \ref{mirrored} we know that $\dim_{\R}M\neq 1$ if and only if $M$ is open and dense. Hence, let us prove that near $z_0$ the set $M$ has empty interior. First note that if $z$ and $w$ are mirrored for $R$, then they are mirrored for $R^n$ as well. Hence, without loss of generality we may assume that $z_0$ is a fixed point of $R$ with a non-real multiplier. By Sullivan's non-wandering component theorem every Fatou component of $R$ is pre-periodic. But since $z_0$ is not mirrored by any other non-repelling periodic point, it follows from the classification of periodic Fatou components that points near $z_0$ can only be mirrored by the points whose orbit converges to $z_0$ or by the points from $\cJ$.

Let us fix $\epsilon>0$ small enough so that the following conditions are fulfilled:
\begin{enumerate}
\item[i)] $R$ is one to one $\overline{\D}(z_0,\epsilon)$ 
\item[ii)] $R(\overline{\D}(z_0,\epsilon))\subset \D(z_0,\epsilon)$
\item[iii)] The prime circles passing through the $z_0$-component of $R^{-1}(\overline{\D}(z_0,\epsilon))$ do not meet any other components of this set.  
\end{enumerate} 
The conditions $ii)$ and $iii)$ imply that all the mirrors of $z\in \D(z_0,\epsilon)$ belong to the union $\overline{\D}(z_0,\epsilon)\cup\cJ$. Next, observe that for $\epsilon>0$ small enough, similarly as in the Example 3.8., the prime circles through  $\overline{\D}(z_0,\epsilon)$ look like parallel lines mapped with $z\to z_0 + \lambda(z-z_0).$ Hence, since $\lambda \notin\R$, two points from $\overline{\D}(z_0,\epsilon)$ can never mirror each other. Therefore, we conclude that the points in $\overline{\D}(z_0,\epsilon)$ can only be mirrored by the points in $\cJ$. However, $z_0$ belongs to the Fatou set. Therefore $\cJ\neq \hC$ has a non-empty interior. This, together with the bound $\#\Phi^{-1}(y)\leq 2d^2$, $y\in\cS$, from Proposition \ref{cardinality}, concludes the proof.
\end{proof}

\noindent Let us explain the meaning of the above statement. A generic rational map is non-exceptional. Moreover, the number of the non-repelling periodic points together with their $\Phi$-fibers is finite. Hence, given an attracting point $z_0$ the above conditions can be obtained by a small perturbation of $R$. Thus, it is reasonable to believe that, as in \cite[Theorem 3.8.]{FP}, the maps with $\dim_{\R}M>1$ can be classified. Therefore we end this section with an open question that we were unable to answer so far: {\it Suppose that $M$ contains an open and dense subset of $\hC$. Does this imply that $R$ is rotationally conjugate to a rational function with real coefficients?}

\section{Proof of Theorem \ref{main1}}

Let us first recall the statement of Theorem \ref{main1} written in the terminology of \S 3.1.

\begin{theorem}\label{main} Let $R$ be a complex map of degree $d\geq 2$ and let $\mu_R$ be its unique measure of maximal entropy on $\hC$. If $R$ is not strongly exceptional, then $\nu_R=\Phi_*(\mu_R)$ is the unique invariant, ergodic measure of maximal entropy $\log d$ on $\cS$. 
\end{theorem}

\begin{proof} As explained in \S 2.1, we only have to prove that $h_{\nu_R}(Q)\geq h_{\mu_R}(R)= \log d$. The rest follows from Proposition \ref{proposition1} and Remark \ref{remark1}.

\medskip

\noindent{\bf Case 1:} Suppose that $\mu_R(S)=0$. Then given a large integer $k>1$, there exists $\delta>0$ such that the $\delta$-neighborhood $N_{\delta}(S)$ of $S$ in $\hC$ satisfies the property $\mu_R(N_{\delta}(S)) < \frac{1}{k}$. Hence, if $z_0\in\cJ$ is a \emph{generic point} in the sense of \cite{Bow}, we have: 
$$
\lim_{n \rightarrow \infty} \frac{\#\{ 0 \le j < n \mid R^j(z) \in N_{\delta}(S) \}}{n} < \frac{1}{k}.
$$
Moreover, we may assume that $z_0\notin N_{\delta}(S)$ and that the orbit of $z_0$ never hits $S$.

Let $x_0=\Phi(z_0)$. Recall that 
$$B(x_0,\epsilon,n)=\{y\in\cS\mid ||Q^k(x_0)-Q^k(y)||<\epsilon,\; 0\leq k< n\}.$$
Given $\ell<n$ we define
$$
B'(n,\ell,\epsilon):= \bigcap_{r=0}^{n-\ell-1}\{w\in \hC\mid ||\Phi(R^r(w)),\Phi(R^{r+\ell}(z_0))||<\epsilon\}.
$$
Since $B'(n,0,\epsilon)=\Phi^{-1}(B(x_0,\epsilon,n))$ this implies that 
$$
\nu_R(B(x_0,\epsilon,n))=\mu_R(B'(n,0,\epsilon)).
$$
Next, observe that $R(B'(n,\ell-1,\epsilon))\subset B'(n,\ell,\epsilon)$. Since $\mu_R$ is invariant this gives 
$$
\mu_R{(B'(n,\ell-1,\epsilon))}\leq \mu_R{(B'(n,\ell,\epsilon))}.
$$ 
We claim that, given $\epsilon>0$ small enough, we have
$$\mu_R(B'(n,0,\epsilon))\leq (2d^2)^{\frac{n}{k}}\left(\frac{1}{d}\right)^{n-\frac{n}{k}}.$$

By Corollary \ref{cover} the set $B'(n,0,\epsilon)$ can be represented as a finite union of connected and pairwise disjoint components $A_j$. Moreover, the maps $\Phi|_{A_j}$ and $R|_{A_j}$ can be assumed to be one-to-one on $A_j$. Furthermore, a similar decomposition exists for any $B'(n,\ell,\epsilon)$ where $R^{\ell}(z_0)\notin  N_{\delta}(S)$. Hence, let $m_0<n$ be the first integer for which $R^{m_0}(z_0)\in  N_{\delta}(S)$. The map $R^{m_0}$ in one-to-one on each $A_j$. Thus by the equidistribution property of $\mu_R$ we have $R^*\mu_R=d\cdot\mu_R$ and hence 
$$\mu_R(A_j)=\frac{\mu_R(R^{m_0}(A_j))}{d^{m_0}}\leq\frac{\mu_R(B'(n,m_0,\epsilon))}{d^{m_0}}.$$
Of course it is possible that several $A_j$'s are mapped into the same component of $B'(n,m_0,\epsilon)$. However, the number of these pieces is bounded by $2d^2$. Thus we have
$$\mu_R(B'(n,0,\epsilon))\leq 2d^2\frac{\mu_R(B'(n,m_0,\epsilon))}{d^{m_0}}.$$
Since $R^{m_0}(z_0)\in  N_{\delta}(S)$ we only have $\mu_R(B'(n,m_0,\epsilon))\leq \mu_R(B'(n,m_0+1,\epsilon))$ at the next step. 
Thus, by a similar argument as above we have
$$\mu_R(B'(n,m_0+1,\epsilon))\leq 2d^2\frac{\mu_R(B'(n,m_1,\epsilon))}{d^{m_1-m_0-1}},$$
where $m_0<m_1<n$ is the next iterate for which $R^{m_1}(z_0)\in  N_{\delta}(S)$. Further, we have
$$\mu_R(B'(n,0,\epsilon))\leq (2d^2)^2\frac{\mu_R(B'(n,m_1+1,\epsilon))}{d^{m_1-1}}.$$
However, note that we have chosen $z_0$ in such a way that $R^{\ell}(z_0)\in  N_{\delta}(S)$ happens in at most $n/k$ cases if $n\in\N$ is sufficiently large. This means that
$$\mu_R(B'(n,0,\epsilon))\leq (2d^2)^{\frac{n}{k}}\frac{\mu_R(B'(n,n-1,\epsilon))}{d^{n-\frac{n}{k}}}\leq (2d^2)^{\frac{n}{k}}\left(\frac{1}{d}\right)^{n-\frac{n}{k}}.$$ 
Thus for large $n\in\N$ and small $\epsilon>0$ we have. 
$$(1-\frac{1}{k})\log d - \frac{1}{k}\log 2d^2 \leq-\frac{1}{n}\log\nu_R(B(x,\epsilon,n)).$$ 
Hence  
$$(1-\frac{1}{k})\log d - \frac{1}{k}\log 2d^2 \leq h_{\nu_R}(Q).$$
Since this holds for every $k>0$ it follows that $\log d \leq h_{\nu_R}(Q)$.

\medskip 
\noindent {\bf Case 2:} Suppose now that $\mu_R(S)>0$. We first prove a general lemma which implies that such a situation is very special.
\begin{lemma}\label{circle}
Let $X\subset \hC$ be a semi-analytic set such that $\dim_{\R}X=1$ and $\mu_R(X)> 0$. Then $\cJ$ is contained in a circle of $\hC$.  
\end{lemma}
\begin{proof} Since $\dim_{\R}X=1$ the set $\textrm{Sing}(X)$ is a discrete set of points. Hence  $\mu_R(\text{Reg}(X))>0$. Therefore there exist a one dimensional irreducible semi-analytic set $L$ in $X$ (i.e. a semi-analytic curve) for which $\mu_R(L)>0$. Recall that if $A$ and $B$ are two semi-analytic curves in $\hC$ and $A\cap B\neq \emptyset$, then either $A\cap B$ is a discrete set of points or else $A\cup B$ is a semi-analytic curve. Let $\tilde{L}$ be the largest possible semi-analytic curve containing $L$. Clearly $\mu_R(\tilde{L})>0$.
Since $R:\hC\rightarrow \hC$ is proper holomorphic map of finite degree, it follows that for every $j\geq 0$ the preimage $R^{-j}(\tilde{L})$ is s semi-analytic set of dimension one and that the set $R^j(\tilde{L})$ is semi-analytic curve (since every sub-analytic set of dimension one is semi-analytic). Thus we have to consider two cases:
\begin{enumerate}
\item for all $j\geq1$ the set $R^{-j}(\tilde{L})\cap \tilde{L}$ is a discrete set of points,
\item there exists $j\geq1$ for which $R^{-j}(\tilde{L})\cap\tilde{L}$ is a semi-analytic curve.
\end{enumerate}

We first prove that in our setting  the case (1) can not happen. Assume the contrary. Then $\mu_R(R^{-j}(\tilde{L})\cap \tilde{L})=0$ for all $j\geq1$. Hence $\mu_R(R^{-j}(\tilde{L})\cap R^{-k}(\tilde{L}))=0$ for all $j\neq k$. Given any $n\in\N$ this implies that  
$$1\geq \mu_R\left(\cup_{j=0}^n R^{-j}(\tilde{L})\right)=\sum_{j=0}^n\mu_R\left(R^{-j}(\tilde{L})\right)=(n+1)\mu_R(\tilde{L})$$
where the last equality follows from the fact that $\mu_R$ is an invariant measure. However, this is impossible unless $\mu_R(\tilde{L})=0$. 

Suppose now that the case $(2)$ is valid. This means that there exists $j\geq1$ for which the set $L_0:=R^{-j}(\tilde{L})\cap\tilde{L}$ is a semi-analytic curve. Moreover, $R^j(L_0)$ is a semi-analytic curve contained in $\tilde{L}$ and since 
$R^j(L_0)\subset R^j(\tilde{L})$ the maximality of $\tilde{L}$ implies that  $R^j(\tilde{L})\subset \tilde{L}$. Thus $\tilde{L}\subset R^{-j}(\tilde{L})$. However, the measure $\mu_R$ is invariant. Thus $\mu_R(R^{-j}(\tilde{L}))=\mu_R(\tilde{L})$ and $\mu_R( R^{-j}(\tilde{L})\backslash \tilde{L}) = \mu_R(\tilde{L}\backslash R^{-j}(\tilde{L}) )=0$. Since $\mu_R$ is ergodic this implies that $\mu_R(\tilde{L})=1$.  

Finally, it follows from the above considerations that $\cJ\cap\text{Reg}(\tilde{L})\neq \emptyset$. Since a semi-analytic set is also locally finite, there exists an open set $U$ for which $\cJ \cap U\subset\text{Reg}(\tilde{L})\cap U$. Thus a relatively open subset of $\cJ$ is contained in a smooth curve. By \cite[Theorem 2]{BE} this implies that $\cJ$ lies in a circle of $\hC$.  
\end{proof}
 
\noindent Let us proceed now with the Case $2$. Since $S$ meets the assumptions of the above lemma, the Julia set of $R$ is contained in an invariant circle $C$ with $\mu_R(C)=1$. Moreover, we know that the circle $C$ is not prime since otherwise $R$ would be strongly exceptional. Therefore, the action of $\Phi$ on the set $C$ is very clear. Indeed, within $C$ every point admits at most one mirror. Equivalently, outside a zero-dimensional semi-analytic set $\textrm{Sing}(\Phi|_C(C))$ the restriction $\Phi|_C$ acts as a two-to-one or a one-to-one cover. Thus, restricting ourselves to $C$ instead of $\hC$ the set of points that needs to be avoided when computing the entropy is
$$\tilde{S}:=\Phi|_C^{-1}(\text{Sing }\Phi|_C(C))\subset C.$$ 
But note that $\mu_R(\tilde{S})=0$. Therefore, we can repeat the above proof word by word using a generic point $z_0\in \cJ$, whose forward orbit avoids $\tilde{S}$, and the one dimensional open sets
$$
B_C'(n,\ell,\epsilon):= \bigcap_{r=0}^{n-\ell-1}\{w\in C\mid ||\Phi(R^r(w)),\Phi(R^{r+\ell}(z_0))||<\epsilon\}.
$$
In particular, we obtain $\nu_R(Q)=\nu_R(Q_C)\geq \log d$, where $Q_C$ denotes the restriction of the map $Q$ to the set $\Phi(C)$.
\end{proof}

\end{document}